\documentclass[12pt]{amsart}
\usepackage{amsmath,amsthm,amsfonts,amssymb,mathrsfs}
\date{\today}

\usepackage{color}

\usepackage{hyperref}

\newtheorem{theorem}{Theorem}[section]

\newtheorem{proposition}[theorem]{Proposition}
\newtheorem{corollary}[theorem]{Corollary}
\newtheorem{lemma}[theorem]{Lemma}
\theoremstyle{definition}
\newtheorem{example}[theorem]{Example}
\newtheorem{remark}[theorem]{Remark}

\begin{document}

\title[On variants of the extended bicyclic semigroup]{On variants of the extended bicyclic semigroup}

\author[Oleg~Gutik and Kateryna Maksymyk]{Oleg~Gutik and Kateryna Maksymyk}
\address{Faculty of Mathematics, National University of Lviv,
Universytetska 1, Lviv, 79000, Ukraine}
\email{oleg.gutik@lnu.edu.ua, ovgutik@yahoo.com, kate.maksymyk15@gmail.com}



\begin{abstract}
In the paper we describe the group $\mathbf{Aut}\left(\mathscr{C}_{\mathbb{Z}}\right)$ of automorphisms of the extended bicyclic semigroup $\mathscr{C}_{\mathbb{Z}}$ and study the variants $\mathscr{C}_{\mathbb{Z}}^{m,n}$ of the extended bicycle semigroup $\mathscr{C}_{\mathbb{Z}}$, where $m,n\in\mathbb{Z}$. In particular, we prove that $\mathbf{Aut}\left(\mathscr{C}_{\mathbb{Z}}\right)$ is isomorphic to the additive group of integers, the extended bicyclic semigroup $\mathscr{C}_{\mathbb{Z}}$ and every its variant are not finitely generated,  and  describe the subset of idempotents $E(\mathscr{C}_{\mathbb{Z}}^{m,n})$ and Green's relations on the semigroup $\mathscr{C}_{\mathbb{Z}}^{m,n}$. Also we show that $E(\mathscr{C}_{\mathbb{Z}}^{m,n})$ is an $\omega$-chain  and any two variants  of the extended bicyclic semigroup $\mathscr{C}_{\mathbb{Z}}$ are isomorphic. At the end we discuss shift-continuous Hausdorff topologies on the variant $\mathscr{C}_{\mathbb{Z}}^{0,0}$. In particular, we prove that if $\tau$ is a Hausdorff shift-continuous topology on $\mathscr{C}_{\mathbb{Z}}^{0,0}$ then each of inequalities $a>0$ or $b>0$ implies that $(a,b)$ is an isolated point of $\big(\mathscr{C}_{\mathbb{Z}}^{0,0},\tau\big)$ and construct an example of a Hausdorff semigroup topology $\tau^*$ on the semigroup $\mathscr{C}_{\mathbb{Z}}^{0,0}$ such that all its points with $ab\leqslant 0$ and $a+b\leqslant 0$ are not isolated in $\big(\mathscr{C}_{\mathbb{Z}}^{0,0},\tau^*\big)$.
\end{abstract}

\maketitle

\section{Introduction and preliminaries}\label{section-0}

We shall follow the terminology of \cite{Carruth-Hildebrant-Koch-1983-1986, Clifford-Preston-1961-1967, Engelking-1989, Ruppert-1984}. In this paper all spaces are assumed to be Hausdorff. By $\mathbb{Z}$, $\mathbb{N}_0$ and $\mathbb{N}$ we denote the sets of all integers, non-negative integers and positive integers, respectively.

A \emph{semigroup} is a non-empty set with a binary associative
operation.

If $S$ is a semigroup, then we shall denote the Green relations on $S$ by $\mathscr{R}$, $\mathscr{L}$, $\mathscr{J}$, $\mathscr{D}$ and $\mathscr{H}$ (see \cite{Clifford-Preston-1961-1967}). For every $a\in S$ by $\mathbf{R}_a$, $\mathbf{L}_a$ and $\mathbf{H}_a$ we denote the $\mathscr{R}$-, $\mathscr{L}$- and $\mathscr{H}$-class in $S$ which contains the element $a$, respectively.   A semigroup $S$ is called \emph{simple} if $S$ does not contain proper two-sided ideals and \emph{bisimple} if $S$ has only one $\mathscr{D}$-class.

If $S$ is a semigroup, then we shall denote the subset of all
idempotents in $S$ by $E(S)$. If
$E(S)$ is closed under multiplication, we shall refer to $E(S)$ a as
\emph{band} (or the \emph{band of} $S$). The semigroup
operation on $S$ determines the following partial order $\preccurlyeq$
on $E(S)$: $e\preccurlyeq f$ if and only if $ef=fe=e$. This order is
called the {\em natural partial order} on $E(S)$. A
\emph{semilattice} is a commutative semigroup of idempotents. A
semilattice $E$ is called {\em linearly ordered} or a \emph{chain}
if its natural order is a linear order. A \emph{maximal chain} of a
semilattice $E$ is a chain which is not properly contained in any other
chain of $E$.

The Axiom of Choice implies the existence of maximal chains in every
partially ordered set. According to
\cite[Definition~II.5.12]{Petrich-1984}, a chain $L$ is called an
$\omega$-chain if $L$ is isomorphic to $\{0,-1,-2,-3,\ldots\}$ with
the usual order $\leqslant$ or equivalently, if $L$ is isomorphic to
$\left(\mathbb{N}_0,\max\right)$.

The \emph{bicyclic semigroup} (or the \emph{bicyclic monoid}) ${\mathscr{C}}(p,q)$ is the semigroup with
the identity $1$ generated by two elements $p$ and $q$ subject
only to the condition $pq=1$. The bicyclic monoid ${\mathscr{C}}(p,q)$ is a combinatorial bisimple $F$-inverse semigroup (see \cite{Lawson-1998}) and it plays an important role in the algebraic theory of semigroups and in the theory of topological semigroups.
For example the well-known O.~Andersen's result~\cite{Andersen-1952}
states that a ($0$--)simple semigroup is completely ($0$--)simple
if and only if it does not contain the bicyclic semigroup. The
bicyclic semigroup cannot be embedded into the stable
semigroups~\cite{Koch-Wallace-1957}.

An interassociate of a semigroup $(S,\cdot)$ is a semigroup $(S,\ast)$ such that for all $a,b,c\in S$, $a\cdot(b\ast c)=(a\cdot b)\ast c$ and $a\ast(b\cdot c)=(a\ast b)\cdot c$. This definition of interassociativity was studied extensively in 1996 by Boyd,  Gould, and Nelson in \cite{Boyd-Gould-Nelson-1998}. Certain classes of semigroups are known to give rise to interassociates with various properties. For example, it is very easy to show that if $S$ is a monoid, every interassociate must satisfy the condition $a\ast b = a\cdot c\cdot b$ for some fixed element $c\in S$ (see \cite{Boyd-Gould-Nelson-1998}). This type of interassociate
was called a \emph{variant} by Hickey \cite{Hickey-1983}. Variants of semigroups of
binary relations have been studied by Chase \cite{Chase-1979, Chase-1979a, Chase-1982}. Variants of transformation semigroups and their representations have been studied \cite{Dolinka-East-2015, Dolinka-East-2018, East-2018, Huang-1996, Mazorchuk-Tsyaputa-2008}.  A general theory of variants has been developed by
a number of authors; see especially \cite{Hickey-1983, Hickey-1986, Khan-Lawson-2001}. For a recent study of variants of finite full transformation
semigroups, and for further references and historical discussion, see \cite{Dolinka-East-2015} and also \cite[Chapter 13]{Ganyushkin-Mazorchuk-2009}. Variants of semilattices were studied in \cite{Desiateryk-2015, Ganyushkin-Desiateryk-2013}. The articles \cite{DolinkaDE-2018??, DolinkaDEHSSS-2018??, DolinkaDEHSSS-2018??a, Dolinka-East-2015} initiated the study of general sandwich semigroups in arbitrary (locally small) categories. In addition, every interassociate of a completely simple semigroup is completely simple \cite{Boyd-Gould-Nelson-1998}. Finally, it is relatively easy to show that every interassociate of a group is isomorphic to the group itself.

In the paper \cite{Givens-Rosin-Linton-2016??} the bicyclic semigroup ${\mathscr{C}}(p,q)$ and its interassociates are investigated. In particular, if $p$ and $q$ are generators of the bicyclic semigroup ${\mathscr{C}}(p,q)$ and $m$ and $n$ are fixed nonnegative integers, the operation $a\ast_{m,n}b=a\cdot q^mp^n\cdot b$ is known to be an interassociate. It was shown that for distinct pairs $(m, n)$ and $(s, t)$, the interassociates $({\mathscr{C}}(p,q),\ast_{m,n})$ and $(\mathscr{C}(p,q),\ast_{s,t})$ are not isomorphic. Also in \cite{Givens-Rosin-Linton-2016??} the authors generalized a result regarding homomorphisms on ${\mathscr{C}}(p,q)$ to homomorphisms on its interassociates.
Later for fixed non-negative integers $m$ and $n$ the interassociate $({\mathscr{C}}(p,q),\ast_{m,n})$ of the bicyclic monoid $\mathscr{C}(p,q)$ will be denoted by $\mathscr{C}_{m,n}$.

A ({\it semi})\emph{topological semigroup} is a topological space with a (separately) continuous semigroup operation. A topology $\tau$ on a semigroup $S$ is called:
\begin{itemize}
  \item \emph{shift-continuous} if $(S,\tau)$ is a semitopological semigroup;
  \item \emph{semigroup}  if $(S,\tau)$ is a topological semigroup.
\end{itemize}

The bicyclic semigroup admits only the discrete semigroup topology and if a topological semigroup $S$ contains it as a dense subsemigroup then ${\mathscr{C}}(p,q)$ is an open subset of $S$~\cite{Eberhart-Selden-1969}. Bertman and  West in \cite{Bertman-West-1976} extend this result for the case of Hausdorff semitopological semigroups. The stable and $\Gamma$-compact topological semigroups do not contain the bicyclic semigroup~\cite{Anderson-Hunter-Koch-1965, Hildebrant-Koch-1986}. The problem of  embedding of the bicyclic monoid into compact-like topological semigroups studied in \cite{Banakh-Dimitrova-Gutik-2009, Banakh-Dimitrova-Gutik-2010, Gutik-Repovs-2007}. Also in the paper \cite{Fihel-Gutik-2011} it was proved that the discrete topology is the unique topology on the extended bicyclic semigroup $\mathscr{C}_{\mathbb{Z}}$ such that the semigroup operation on $\mathscr{C}_{\mathbb{Z}}$ is separately continuous. Amazing dichotomy for the bicyclic monoid with adjoined zero $\mathscr{C}^0={\mathscr{C}}(p,q)\sqcup\{0\}$ was proved in \cite{Gutik-2015}: every Hausdorff locally compact semitopological bicyclic semigroup with adjoined zero $\mathscr{C}^0$ is either compact or discrete.

In the paper \cite{Gutik-Maksymyk-2016} we studied semitopological interassociates $({\mathscr{C}}(p,q),\ast_{m,n})$ of the bicyclic monoid $\mathscr{C}(p,q)$ for arbitrary non-negative integers $m$ and $n$. Some results from \cite{Bertman-West-1976, Eberhart-Selden-1969, Gutik-2015} obtained for the bicyclic semigroup to its interassociate $({\mathscr{C}}(p,q),\ast_{m,n})$ were extended. In particular, we showed that for arbitrary non-negative integers $m$, $n$ and every Hausdorff topology $\tau$ on $\mathscr{C}_{m,n}$ such that $\left(\mathscr{C}_{m,n},\tau\right)$ is a semitopological semigroup, is discrete. Also, we proved that if an  interassociate of the bicyclic monoid $\mathscr{C}_{m,n}$ is a dense subsemigroup of a Hausdorff semitopological semigroup $(S,\cdot)$ and $I=S\setminus\mathscr{C}_{m,n}\neq\varnothing$ then $I$ is a two-sided ideal of the semigroup $S$ and show that for arbitrary non-negative integers $m$, $n$, any Hausdorff locally compact semitopological semigroup $\mathscr{C}_{m,n}^0$ ($\mathscr{C}_{m,n}^0=\mathscr{C}_{m,n}\sqcup\{0\}$) is either discrete or compact.

On the Cartesian
product $\mathscr{C}_{\mathbb{Z}}=\mathbb{Z}\times\mathbb{Z}$ we
define the semigroup operation as follows:
\begin{equation}\label{eq-0.1}
    (a,b)\cdot(c,d)=
\left\{
  \begin{array}{ll}
    (a-b+c,d), & \hbox{if }~b<c; \\
    (a,d),     & \hbox{if }~b=c; \\
    (a,d+b-c), & \hbox{if }~b>c,\\
  \end{array}
\right.
\end{equation}
for $a,b,c,d\in\mathbb{Z}$. The set $\mathscr{C}_{\mathbb{Z}}$ with
such defined operation will be called the \emph{extended bicyclic semigroup}~\cite{Warne-1968}.

In the paper \cite{Fihel-Gutik-2011} algebraic properties of $\mathscr{C}_{\mathbb{Z}}$ were described and it was proved therein that every non-trivial
congruence $\mathfrak{C}$ on the semigroup
$\mathscr{C}_{\mathbb{Z}}$ is a group congruence, and moreover the
quotient semigroup $\mathscr{C}_{\mathbb{Z}}/\mathfrak{C}$ is
isomorphic to a cyclic group. Also it was shown that the semigroup
$\mathscr{C}_{\mathbb{Z}}$ as a Hausdorff semitopological semigroup
admits only the discrete topology and the closure
$\operatorname{cl}_T\left(\mathscr{C}_{\mathbb{Z}}\right)$ of the
semigroup $\mathscr{C}_{\mathbb{Z}}$ in a topological semigroup $T$ was studied.

In this paper we describe the group $\mathbf{Aut}\left(\mathscr{C}_{\mathbb{Z}}\right)$ of automorphisms of the extended bicyclic semigroup $\mathscr{C}_{\mathbb{Z}}$ and study a variant $\mathscr{C}_{\mathbb{Z}}^{m,n}=\left(\mathscr{C}_{\mathbb{Z}},\ast_{m,n}\right)$ of the extended bicycle semigroup $\mathscr{C}_{\mathbb{Z}}$, where $m,n\in\mathbb{Z}$, which is defined by the formula
\begin{equation}\label{eq-0.2}
  (a,b)\ast_{m,n}(c,d)=(a,b)\cdot(m,n)\cdot(c,d).
\end{equation}
In particular, we prove that $\mathbf{Aut}\left(\mathscr{C}_{\mathbb{Z}}\right)$ is isomorphic to the additive group of integers, the extended bicyclic semigroup $\mathscr{C}_{\mathbb{Z}}$ and every its variant are not finitely generated, and describe the subset of idempotents $E(\mathscr{C}_{\mathbb{Z}}^{m,n})$ and Green's relations on the semigroup $\mathscr{C}_{\mathbb{Z}}^{m,n}$. Also we show that $E(\mathscr{C}_{\mathbb{Z}}^{m,n})$ is an $\omega$-chain  and any two variants  of the extended bicyclic semigroup $\mathscr{C}_{\mathbb{Z}}$ are isomorphic. At the end we discuss shift-continuous Hausdorff topologies on the variant $\mathscr{C}_{\mathbb{Z}}^{0,0}$. In particular, we prove that if $\tau$ is a Hausdorff shift-continuous topology on $\mathscr{C}_{\mathbb{Z}}^{0,0}$ then each of inequalities $a>0$ or $b>0$ implies that $(a,b)$ is an isolated point of $\big(\mathscr{C}_{\mathbb{Z}}^{0,0},\tau\big)$ and constructe an example a Hausdorff semigroup topology $\tau^*$ on the semigroup $\mathscr{C}_{\mathbb{Z}}^{0,0}$ such that all its points with the properties $ab\leqslant 0$ and $a+b\leqslant 0$ are not isolated in $\big(\mathscr{C}_{\mathbb{Z}}^{0,0},\tau^*\big)$.


\section{On the group of automorphisms of the extended bicyclic semigroup}\label{section-1}

\begin{lemma}\label{lemma-1.1}
For arbitrary integer $k$ the set $$\mathscr{C}_{\mathbb{Z}}^{\geqslant k}=\left\{(i,j)\in\mathscr{C}_{\mathbb{Z}}\colon i,j\geqslant k\right\}$$ with the induced from $\mathscr{C}_{\mathbb{Z}}$ semigroup operation is isomorphic to the bicyclic semigroup $\mathscr{C}(p,q)$ by the mapping $h_k\colon \mathscr{C}(p,q)\to \mathscr{C}_{\mathbb{Z}}^{\geqslant k}$, $q^ip^j\mapsto (i+k,j+k)$.
\end{lemma}

\begin{proof} 
Since
\begin{equation*}
\begin{split}
  h_k\left(q^mp^n\cdot q^ip^j\right) &
=
\left\{
  \begin{array}{ll}
    h_k\left(q^{m-n+i}p^j\right), & \hbox{if~} m\leqslant i;\\
    h_k\left(q^np^{m-i+j}\right), & \hbox{if~} m\geqslant i
  \end{array}
\right.
=   \\
    & =
\left\{
  \begin{array}{ll}
   (m-n+i+k,j+k), & \hbox{if~} m\leqslant i;\\
   (n+k,m-i+j+k), & \hbox{if~} m\geqslant i
  \end{array}
\right.
\end{split}
\end{equation*}
and
\begin{equation*}
\begin{split}
  h_k\left(q^mp^n\right)\cdot h_k\left(q^ip^j\right) & =(m+k,n+k)\cdot (i+k,j+k)= \\
    & =
\left\{
  \begin{array}{ll}
   (m-n+i+k,j+k), & \hbox{if~} m\leqslant i;\\
   (n+k,m-i+j+k), & \hbox{if~} m\geqslant i.
  \end{array}
\right.
\end{split}
\end{equation*}
for any $q^mp^n, q^ip^j\in \mathscr{C}(p,q)$ we have that the so defined map $h_k\colon \mathscr{C}(p,q)\to \mathscr{C}_{\mathbb{Z}}^{\geqslant k}$ is a homomorphism, and simple verifications imply that it is a bijection.
\end{proof}

\begin{theorem}\label{theorem-1.2}
For an arbitrary integer $k$ the map $h_k\colon \mathscr{C}_{\mathbb{Z}}\to \mathscr{C}_{\mathbb{Z}}$ defined by by the formula
\begin{equation}\label{eq-1.1}
  h_k\left((i,j)\right)=(i+k,j+k),
\end{equation}
is an automorphism of the extended bicyclic semigroup $\mathscr{C}_{\mathbb{Z}}$ and every automorphism $\mathfrak{h}\colon \mathscr{C}_{\mathbb{Z}}\to \mathscr{C}_{\mathbb{Z}}$  of $\mathscr{C}_{\mathbb{Z}}$ has form \eqref{eq-1.1}. Moreover, the group $\mathbf{Aut}\left(\mathscr{C}_{\mathbb{Z}}\right)$ of automorphisms of $\mathscr{C}_{\mathbb{Z}}$ is isomorphic to the additive group of integers $\mathbb{Z}(+)$ and this isomorphism $\mathfrak{H}\colon \mathbb{Z}(+)\to \mathbf{Aut}\left(\mathscr{C}_{\mathbb{Z}}\right)$ is defined by the formula $\mathfrak{H}(k)=h_k$, $k\in\mathbb{Z}$.
\end{theorem}

\begin{proof}
For any $(m,n), (i,j)\in \mathscr{C}_{\mathbb{Z}}$ we have that
\begin{equation*}
\begin{split}
  h_k\left((m,n)\cdot(i,j)\right) & =
\left\{
  \begin{array}{ll}
    h_k\left((m-n+i,j)\right), & \hbox{if~} m\leqslant i;\\
    h_k\left((n,m-i+j)\right), & \hbox{if~} m\geqslant i
  \end{array}
\right.
= \\
    & =
\left\{
  \begin{array}{ll}
   (m-n+i+k,j+k), & \hbox{if~} m\leqslant i;\\
   (n+k,m-i+j+k), & \hbox{if~} m\geqslant i
  \end{array}
\right.
\end{split}
\end{equation*}
and
\begin{equation*}
\begin{split}
  h_k\left((m,n)\right)\cdot h_k\left((i,j)\right) & =(m+k,n+k)\cdot (i+k,j+k)= \\
    & =
\left\{
  \begin{array}{ll}
   (m-n+i+k,j+k), & \hbox{if~} m\leqslant i;\\
   (n+k,m-i+j+k), & \hbox{if~} m\geqslant i.
  \end{array}
\right.
\end{split}
\end{equation*}
Simple verifications imply that for every integer $k$ the so defined map $h_k$ is a bijection, and hence it is an automorphism of the extended bicyclic semigroup $\mathscr{C}_{\mathbb{Z}}$.

Let $\mathfrak{h}\colon \mathscr{C}_{\mathbb{Z}}\to \mathscr{C}_{\mathbb{Z}}$  be an arbitrary automorphism of $\mathscr{C}_{\mathbb{Z}}$. Since $(0,0)$ is an idempotent of $\mathscr{C}_{\mathbb{Z}}$, $\mathfrak{h}\left((0,0)\right)$ is an idempotent of $\mathscr{C}_{\mathbb{Z}}$ as well, and hence by Proposition~2.1$(i)$ from \cite{Fihel-Gutik-2011} we have that $\mathfrak{h}\left((0,0)\right)=(k,k)$ for some integer $k$. Since $(1,1)$ is the maximum of the subset
$$
\left\{(n,n)\in E(\mathscr{C}_{\mathbb{Z}})\colon (n,n)\preccurlyeq(0,0)\right\}\setminus\left\{(0,0)\right\}
$$
of the poset $\left(E(\mathscr{C}_{\mathbb{Z}}),\preccurlyeq\right)$ and $\mathfrak{h}\colon \mathscr{C}_{\mathbb{Z}}\to \mathscr{C}_{\mathbb{Z}}$ is an automorphism of $\mathscr{C}_{\mathbb{Z}}$ we get that $\mathfrak{h}\left((1,1)\right)=(k+1,k+1)$ because $(k+1,k+1)$ is the maximum of the subset
$$
\left\{(n,n)\in E(\mathscr{C}_{\mathbb{Z}})\colon (n,n)\preccurlyeq(k,k)=\mathfrak{h}\left((0,0)\right)\right\}\setminus\left\{(k,k)\right\}
$$
of the poset $\left(E(\mathscr{C}_{\mathbb{Z}}),\preccurlyeq\right)$. Then by induction we obtain that $\mathfrak{h}\left((i,i)\right)=(i+k,i+k)$ for every positive integer $i$. Also, since $(-1,-1)$ is the minimum of the subset
$$
\left\{(n,n)\in E(\mathscr{C}_{\mathbb{Z}})\colon (0,0)\preccurlyeq(n,n)\right\}\setminus\left\{(0,0)\right\}
$$
of the poset $\left(E(\mathscr{C}_{\mathbb{Z}}),\preccurlyeq\right)$ and $\mathfrak{h}\colon \mathscr{C}_{\mathbb{Z}}\to \mathscr{C}_{\mathbb{Z}}$ is an automorphism of $\mathscr{C}_{\mathbb{Z}}$ we obtain that $\mathfrak{h}\left((-1,-1)\right)=(k-1,k-1)$ because $(k-1,k-1)$ is the minimum of the subset
$$
\left\{(n,n)\in E(\mathscr{C}_{\mathbb{Z}})\colon (k,k)=\mathfrak{h}\left((0,0)\right)\preccurlyeq(n,n)\right\}\setminus\left\{(k,k)\right\}
$$
of the poset $\left(E(\mathscr{C}_{\mathbb{Z}}),\preccurlyeq\right)$. Then by induction we get that $\mathfrak{h}\left((-i,-i)\right)=(-i+k,-i+k)$ for every positive integer $i$.

Since $\mathfrak{h}\colon \mathscr{C}_{\mathbb{Z}}\to \mathscr{C}_{\mathbb{Z}}$ is an automorphism of $\mathscr{C}_{\mathbb{Z}}$, $\mathscr{C}_{\mathbb{Z}}$ is an inverse semigroup and by Proposition~2.1$(iv)$ of \cite{Fihel-Gutik-2011} every $\mathscr{H}$-class in $\mathscr{C}_{\mathbb{Z}}$ is a singleton, the equalities
\begin{equation*}
  \mathbf{L}_{(i,j)}=\mathbf{L}_{(j,j)}, \qquad  \mathbf{R}_{(i,j)}=\mathbf{R}_{(i,i)} \qquad \hbox{and} \qquad \mathbf{H}_{(i,j)}=\mathbf{L}_{(i,j)}\cap\mathbf{R}_{(i,j)}
\end{equation*}
imply that
\begin{equation*}
  \mathbf{L}_{\mathfrak{h}((i,j))}=\mathbf{L}_{\mathfrak{h}((j,j))}, \qquad  \mathbf{R}_{\mathfrak{h}((i,j))}=\mathbf{R}_{\mathfrak{h}((i,i))} \qquad \hbox{and} \qquad \mathbf{H}_{\mathfrak{h}((i,j))}=\mathbf{L}_{\mathfrak{h}((i,j))}\cap\mathbf{R}_{\mathfrak{h}((i,j))},
\end{equation*}
and hence we have that
\begin{equation*}
  \left\{\mathfrak{h}((i,j))\right\}=\mathbf{H}_{\mathfrak{h}((i,j))}=\mathbf{L}_{\mathfrak{h}((i,j))}\cap\mathbf{R}_{\mathfrak{h}((i,j))}= \mathbf{L}_{(i+k,j+k)}\cap\mathbf{R}_{(i+k,j+k)}=\left\{(i+k,j+k)\right\},
\end{equation*}
for all integers $i$ and $j$. This completes the proof of the first statement of the theorem.

For arbitrary integers $k_1$ and $k_2$ we have that
\begin{equation*}
\begin{split}
  \left(h_{k_1}\circ h_{k_2}\right)(i,j) & = h_{k_1}\left( h_{k_2}\left((i,j)\right)\right)= \\
    & =h_{k_1}\left((i+k_2,j+k_2)\right)=\\
    & =(i+k_2+k_1,j+k_2+k_1)=\\
    & = h_{k_1+k_2}(i,j),
\end{split}
\end{equation*}
$h_0\colon \mathscr{C}_{\mathbb{Z}}\to \mathscr{C}_{\mathbb{Z}}$, $(i,j)\mapsto(i,j)$ is the identity automorphism of $\mathscr{C}_{\mathbb{Z}}$ and $h_{-k_1}\colon \mathscr{C}_{\mathbb{Z}}\to \mathscr{C}_{\mathbb{Z}}$, $(i,j)\mapsto(i-k_1,j-k_1)$ is the converse map to the map $h_{k_1}\colon \mathscr{C}_{\mathbb{Z}}\to \mathscr{C}_{\mathbb{Z}}$. This completes the proof of the second statement of the theorem.
\end{proof}

Serhii Bardyla asked the following question  on the Seminar on S-act Theory and Spectral Spaces at Lviv University.

\medskip
\noindent
\textbf{Question.}
\emph{Are the semigroups $\mathscr{C}_{\mathbb{Z}}$ and $\mathscr{C}_{\mathbb{Z}}^{m,n}$, $m,n\in\mathbb{Z}$, finitely generated?}
\medskip

Later we shall give a negative answer to this  question.

\begin{lemma}\label{lemma-1.4}
For every finite subset $F=\left\{(i_1,j_1),\ldots,(i_n,j_n)\right\}$ of the extended bicyclic semigroup $\mathscr{C}_{\mathbb{Z}}$ there exists a subsemigroup $S$ of $\mathscr{C}_{\mathbb{Z}}$ such that $S$ is isomorphic to the bicyclic semigroup and $S$ contains the semigroup $\left\langle F\right\rangle$ which is generated by the set $F$. Moreover $\left\langle F\right\rangle$ is a subsemigroup of $\mathscr{C}_{\mathbb{Z}}^{\geqslant k}$, where $k=\min\left\{i_1,j_1,\ldots,i_n,j_n\right\}$.
\end{lemma}

\begin{proof}
By Lemma~\ref{lemma-1.1}, $\mathscr{C}_{\mathbb{Z}}^{\geqslant k}$ is an inverse subsemigroup of $\mathscr{C}_{\mathbb{Z}}$ for any integer $k$ and formula \eqref{eq-0.1} implies that $\left\langle F\right\rangle$ is a subsemigroup of $\mathscr{C}_{\mathbb{Z}}^{\geqslant k}$ for $k=\min\left\{i_1,j_1,\ldots,i_n,j_n\right\}$
\end{proof}

The following  theorem is a consequence of Lemma~\ref{lemma-1.4}.

\begin{theorem}\label{theorem-1.5}
The extended bicyclic semigroup $\mathscr{C}_{\mathbb{Z}}$ is not finitely generated as an inverse semigroup.
\end{theorem}

\begin{corollary}\label{corollary-1.6}
The extended bicyclic semigroup $\mathscr{C}_{\mathbb{Z}}$ is not finitely generated as a semigroup.
\end{corollary}


\section{Algebraic properties of the semigroup $\mathscr{C}_{\mathbb{Z}}^{m,n}$}\label{section-2}

Since a semigroup $S$ is simple if and only if $SsS=S$ for every $s\in S$ we have that $S(csc)S=S$ for all $s,c\in S$. Since the extended bicycle semigroup $\mathscr{C}_{\mathbb{Z}}$ is simple, the above arguments imply the following property of the semigroup $\mathscr{C}_{\mathbb{Z}}^{m,n}$.

\begin{proposition}\label{proposition-2.1}
$\mathscr{C}_{\mathbb{Z}}^{m,n}$ is a simple semigroup for all $m,n\in\mathbb{Z}$.
\end{proposition}

\begin{proposition}\label{proposition-2.2}
Let $m$ and $n$ be arbitrary integers. Then an element $(a,b)$ of the semigroup $\mathscr{C}_{\mathbb{Z}}^{m,n}$ is an idempotent if and only if $(a,b)=(n+i,m+i)$ for some $i\in \mathbb{N}_0$.
\end{proposition}

\begin{proof} 
$(\Leftarrow)$  Suppose that $a=n+i$ and $b=m+i$ for some $i\in \mathbb{N}_0$. Then
\begin{equation*}
\begin{split}
  (a,b)\ast_{m,n}(a,b) & =(n+i,m+i)\cdot(m,n)\cdot(n+i,m+i)= \\
    & =(n+i,m+i-m+n)\cdot(n+i,m+i)=\\
    & =(n+i,i+n)\cdot(n+i,m+i)=\\
    & =(n+i,m+i)=\\
    & =(a,b).
\end{split}
\end{equation*}

$(\Rightarrow)$ Formulae \eqref{eq-0.1}, \eqref{eq-0.2}, and items $(ix)$ and $(x)$ of Proposition~2.1~\cite{Fihel-Gutik-2011} imply that for any element $(a,b)$ of the semigroup $\mathscr{C}_{\mathbb{Z}}^{m,n}$ we have that
\begin{equation}\label{eq-2.1}
\begin{split}
  (a,b)\ast_{m,n}\mathscr{C}_{\mathbb{Z}}^{m,n} & = (a,b)\cdot (m,n) \cdot\mathscr{C}_{\mathbb{Z}}=\\
  & =
  \left\{
    \begin{array}{cl}
      (a,b-m+n)\cdot\mathscr{C}_{\mathbb{Z}}, & \hbox{if~} b\geqslant m;\\
      (a-b+m,n)\cdot\mathscr{C}_{\mathbb{Z}}, & \hbox{if~} b< m
    \end{array}
    \right.
    =\\
   & =
   \left\{
    \begin{array}{cl}
      \left\{(x,y)\in\mathscr{C}_{\mathbb{Z}}^{m,n}\colon x\geqslant a\right\}, & \hbox{if~} b\geqslant m;\\
      \left\{(x,y)\in\mathscr{C}_{\mathbb{Z}}^{m,n}\colon x\geqslant a-b+m\right\}, & \hbox{if~} b< m
    \end{array}
    \right.
\end{split}
\end{equation}
and
\begin{equation}\label{eq-2.2}
\begin{split}
\mathscr{C}_{\mathbb{Z}}^{m,n}\ast_{m,n}(a,b) & = \mathscr{C}_{\mathbb{Z}}\cdot (m,n) \cdot(a,b)=\\
  & =
  \left\{
    \begin{array}{cl}
      \mathscr{C}_{\mathbb{Z}}\cdot(a-n+m,b), & \hbox{if~} a\geqslant n;\\
      \mathscr{C}_{\mathbb{Z}}\cdot (m,n-a+b), & \hbox{if~} a< n
    \end{array}
    \right.
    =\\
   & =
   \left\{
    \begin{array}{cl}
      \left\{(x,y)\in\mathscr{C}_{\mathbb{Z}}^{m,n}\colon y\geqslant b\right\}, & \hbox{if~} a\geqslant n;\\
      \left\{(x,y)\in\mathscr{C}_{\mathbb{Z}}^{m,n}\colon y\geqslant n-a+b\right\}, & \hbox{if~} a< n.
    \end{array}
    \right.
\end{split}
\end{equation}
Since
\begin{equation*}
(a,b)=(a,b)\ast_{m,n}(a,b)\subseteq (a,b)\ast_{m,n}\mathscr{C}_{\mathbb{Z}}^{m,n}\cap \mathscr{C}_{\mathbb{Z}}^{m,n}\ast_{m,n}(a,b),
\end{equation*}
formulae \eqref{eq-2.1}, \eqref{eq-2.2} imply that $b\geqslant m$ and $a\geqslant n$. Then
\begin{equation*}
\begin{split}
  (a,b)\ast_{m,n}(a,b)&=(a,b)\cdot (m,n) \cdot(a,b)=\\
   &=(a,b-m+n)\cdot(a,b)=\\
   &=
\left\{
    \begin{array}{cl}
      (2a-b-n+m,b), & \hbox{if~} a\geqslant b-m+n;\\
      (a,2b-a-m+n), & \hbox{if~} a< b-m+n
    \end{array}
    \right.
\end{split}
\end{equation*}
and hence the equality $(a,b)\ast_{m,n}(a,b)=(a,b)$ implies that $a-b=n-m$. Since $a$ and $b$ are integers such that $b\geqslant m$ and $a\geqslant n$ the equality $a-b=n-m$ implies that $(a,b)=(n+i,m+i)$ for some non-negative integer $i$.
\end{proof}

Since $E(\mathscr{C}_{\mathbb{Z}}^{m,n})=\left\{(n+i,m+i)\colon i\in\mathbb{N}_0\right\}$ we denote the idempotent $(n+i,m+i)$ of $\mathscr{C}_{\mathbb{Z}}^{m,n}$ by $e_i$ for arbitrary $i\in\mathbb{N}_0$.

\begin{lemma}\label{lemma-2.3}
Let $m$ and $n$ be arbitrary integers. Then $e_i\preccurlyeq e_j$ in $E(\mathscr{C}_{\mathbb{Z}}^{m,n})$ if and only if $j\leqslant i$ and hence $E(\mathscr{C}_{\mathbb{Z}}^{m,n})$ is an $\omega$-chain.
\end{lemma}

\begin{proof} 
If $e_i\preccurlyeq e_j$ in $E(\mathscr{C}_{\mathbb{Z}}^{m,n})$ then
\begin{equation*}
\begin{split}
  e_i\ast_{m,n} e_j & =(n+i,m+i)\cdot(m,n)\cdot(n+j,m+j)= \\
    & =(n+i,n+i)\cdot(n+j,m+j)= \\
    & =(n+i,m+i)= \\
    & = e_i
\end{split}
\end{equation*}
and
\begin{equation*}
\begin{split}
  e_j\ast_{m,n} e_i & =(n+j,m+j)\cdot(m,n)\cdot(n+i,m+i)= \\
    & =(n+j,n+j)\cdot(n+i,m+i)= \\
    & =(n+i,m+i)=  \\
    & =e_i
\end{split}
\end{equation*}
imply that $j\leqslant i$. The converse statement follows from the semigroup operation of $\mathscr{C}_{\mathbb{Z}}^{m,n}$.

An isomorphism $\varphi\colon E(\mathscr{C}_{\mathbb{Z}}^{m,n})\to \left(\mathbb{N}_0,\max\right)$ we define by the formula $\varphi(e_i)=i$, $i\in \mathbb{N}_0$.
\end{proof}

The following proposition describes Green's relations  on the semigroup $\mathscr{C}_{\mathbb{Z}}^{m,n}$.

\begin{proposition}\label{proposition-2.4}
Let $m$ and $n$ be arbitrary integers, $(a,b)$ and $(c,d)$ be elements of $\mathscr{C}_{\mathbb{Z}}^{m,n}$. Then the following statements hold.
\begin{enumerate}
  \item\label{proposition-2.4-1} $(a,b)\mathscr{R}(c,d) \; \Longleftrightarrow \;(a=c)\wedge\left((b=d)\vee(b,d\geqslant m)\right)$.
  \item\label{proposition-2.4-2} $(a,b)\mathscr{L}(c,d) \; \Longleftrightarrow \;(b=d)\wedge\left((a=c)\vee(a,c\geqslant n)\right)$.
  \item\label{proposition-2.4-3} $(a,b)\mathscr{H}(c,d) \; \Longleftrightarrow \;(a,b)=(c,d)$.
  \item\label{proposition-2.4-4} $(a,b)\mathscr{D}(c,d) \; \Longleftrightarrow \;(a,b)=(c,d)\vee(a,c\geqslant n)\vee(b,d\geqslant m)$.
  \item\label{proposition-2.4-5} $(a,b)\mathscr{J}(c,d)$ for all $(a,b),(c,d)\in\mathscr{C}_{\mathbb{Z}}^{m,n}$.
\end{enumerate}
\end{proposition}

\begin{proof}
By formula \eqref{eq-2.1}  we get that
\begin{equation*}
  \left\{(a,b)\right\}\cup (a,b)\ast_{m,n}\mathscr{C}_{\mathbb{Z}}^{m,n}=
\left\{
    \begin{array}{cl}
      \left\{(x,y)\in\mathscr{C}_{\mathbb{Z}}^{m,n}\colon x\geqslant a\right\}, & \hbox{if~} b\geqslant m;\\
      \left\{(a,b)\right\}\cup\left\{(x,y)\in\mathscr{C}_{\mathbb{Z}}^{m,n}\colon x\geqslant a-b+m\right\}, & \hbox{if~} b< m.
    \end{array}
    \right.
\end{equation*}
The above formula implies statement \eqref{proposition-2.4-1}. The proof of statement \eqref{proposition-2.4-2} is similar.

Statement \eqref{proposition-2.4-3} follows from \eqref{proposition-2.4-1} and \eqref{proposition-2.4-2}.

For the proof of assertion \eqref{proposition-2.4-4} we consider the following three cases.
\begin{itemize}
  \item[$(i)$] If $a<n$ and $b<m$ then by statements \eqref{proposition-2.4-1} and \eqref{proposition-2.4-2} we have that $$(x,y)\mathscr{R}(a,b) \; \Longleftrightarrow \; (x,y)=(a,b)$$ and $$(x,y)\mathscr{L}(a,b) \; \Longleftrightarrow \; (x,y)=(a,b),$$  for $(x,y)\in\mathscr{C}_{\mathbb{Z}}^{m,n}$, and hence in this case we get that $$(c,d)\mathscr{R}(a,b) \; \Longleftrightarrow \; (c,d)=(a,b).$$

  \item[$(ii)$] If $a\geqslant n$ then by statements \eqref{proposition-2.4-1} and \eqref{proposition-2.4-2} we obtain that the $\mathscr{L}$-class $\mathbf{L}_{(a,b)}$ of the element $(a,b)$ intersects the $\mathscr{R}$-class  $\mathbf{R}_{(x,y)}$ of any element $(x,y)$ with $y\geqslant m$.
  \item[$(iii)$] If $b\geqslant m$ then by statements \eqref{proposition-2.4-1} and \eqref{proposition-2.4-2} we obtain that the $\mathscr{R}$-class $\mathbf{R}_{(a,b)}$ of the element $(a,b)$ intersects the $\mathscr{L}$-class  $\mathbf{L}_{(x,y)}$ of any element $(x,y)$ with $x\geqslant n$.
\end{itemize}
Thus, in the case of $(ii)$ or $(iii)$ we have that
$$
(a,b)\mathscr{D}(c,d) \; \Longleftrightarrow \;(a,c\geqslant n)\vee(b,d\geqslant m)
$$
which with case $(i)$ implies statement \eqref{proposition-2.4-4}.

Proposition~\ref{proposition-2.1} implies statement \eqref{proposition-2.4-5}.
\end{proof}

\begin{lemma}\label{lemma-2.5}
For arbitrary idempotents $(i,i)$ and $(j,j)$ of the extended bicyclic semigroup $\mathscr{C}_{\mathbb{Z}}$ variants $\mathscr{C}_{\mathbb{Z}}^{i,i}$  and $\mathscr{C}_{\mathbb{Z}}^{j,j}$ are isomorphic.
\end{lemma}

\begin{proof}
By Theorem~\ref{theorem-1.2} for every positive integer $k$ the map
$$
h_k\colon \mathscr{C}_{\mathbb{Z}}\to \mathscr{C}_{\mathbb{Z}}, \qquad (i,j)\mapsto (i+k,j+k)
$$
is an automorphism of the extended bicyclic semigroup $\mathscr{C}_{\mathbb{Z}}$. This implies that the map $h_k$ determines the isomorphism $\mathfrak{h}_k\colon \mathscr{C}_{\mathbb{Z}}^{0,0}\to\mathscr{C}_{\mathbb{Z}}^{k,k}$ of variants $\mathscr{C}_{\mathbb{Z}}^{0,0}$  and $\mathscr{C}_{\mathbb{Z}}^{k,k}$ for every integer $k$. Indeed, we put $\mathfrak{h}_k\left((a,b)\right)=h_k\left((a,b)\right)$ for each $(a,b)\in\mathscr{C}_{\mathbb{Z}}^{0,0}$. Then
\begin{equation*}
\begin{split}
  \mathfrak{h}_k\left((a,b)\ast_{(0,0)}(c,d)\right) & =h_k\left((a,b)\ast_{(0,0)}(c,d)\right)=\\
    & = h_k\left((a,b)\cdot(0,0)\cdot(c,d)\right)= \\
    & = h_k\left((a,b)\right)\cdot h_k\left((0,0)\right)\cdot h_k\left((c,d)\right)= \\
    & = (a+k,b+k)\cdot(k,k)\cdot(c+k,d+k)= \\
    & = (a+k,b+k)\ast_{(k,k)}(c+k,d+k)=\\
    & = h_k\left((a,b)\right)\ast_{(k,k)}h_k\left((c,d)\right)=  \\
    & = \mathfrak{h}_k\left((a,b)\right)\ast_{(k,k)}\mathfrak{h}_k\left((c,d)\right),
\end{split}
\end{equation*}
for arbitrary $(a,b), (c,d)\in \mathscr{C}_{\mathbb{Z}}^{0,0}$. Since for any positive integer $k$ the map
$$
h_k\colon \mathscr{C}_{\mathbb{Z}}\to \mathscr{C}_{\mathbb{Z}}, \qquad (i,j)\mapsto (i+k,j+k)
$$
as a self-mapping of the set $\mathscr{C}_{\mathbb{Z}}$ is bijective, we conclude that $\mathfrak{h}_k\colon \mathscr{C}_{\mathbb{Z}}^{0,0}\to\mathscr{C}_{\mathbb{Z}}^{k,k}$ is an isomorphism of variants $\mathscr{C}_{\mathbb{Z}}^{0,0}$  and $\mathscr{C}_{\mathbb{Z}}^{k,k}$. This completes the proof of the lemma.
\end{proof}

\begin{lemma}\label{lemma-2.6}
For an arbitrary integer $r$ and an arbitrary positive integer $p$ the variants $\mathscr{C}_{\mathbb{Z}}^{r,r}$  and $\mathscr{C}_{\mathbb{Z}}^{r+p,r}$ of the extended bicyclic semigroup $\mathscr{C}_{\mathbb{Z}}$ are isomorphic.
\end{lemma}

\begin{proof} 
Fix an arbitrary integer $r$ and an arbitrary positive integer $p$. We define a map $\mathfrak{h}\colon \mathscr{C}_{\mathbb{Z}}^{r,r}\to\mathscr{C}_{\mathbb{Z}}^{r+p,r}$ by the formula
$$
\mathfrak{h}\left((r+i,r+j)\right)=(r+i,r+j+p).
$$
Then for arbitrary $(r+i,r+j),(r+k,r+l)\in\mathscr{C}_{\mathbb{Z}}^{r,r}$ we have that
\begin{equation*}
\begin{split}
  \mathfrak{h}\left((r+i,r+j)\right. &{}\ast_{r,r} \left.(r+k,r+l)\right) = \mathfrak{h}\left((r+i,r+j)\cdot(r,r)\cdot(r+k,r+l)\right)=\\
    & =
\left\{
  \begin{array}{ll}
    \mathfrak{h}\left((r+i-j,r)\cdot(r+k,r+l)\right),   & \hbox{if~~} r+j\leqslant r;\\
    \mathfrak{h}\left((r+i,r+j)\cdot(r+k,r+l)\right),   & \hbox{if~~} r+j>r
  \end{array}
\right.
=\\
    & =
\left\{
  \begin{array}{llll}
    \mathfrak{h}\left((r+i-j,r+l-k)\right), & \hbox{if~~} r+j\leqslant r & \hbox{and}  & r+k\leqslant r;\\
    \mathfrak{h}\left((r+i-j+k,r+l)\right), & \hbox{if~~} r+j\leqslant r & \hbox{and}  & r+k>r;\\
    \mathfrak{h}\left((r+i,r+j-k+l)\right), & \hbox{if~~} r+j>r          & \hbox{and}  & r+k\leqslant r+j;\\
    \mathfrak{h}\left((r+i-j+k,r+l)\right), & \hbox{if~~} r+j>r          & \hbox{and}  & r+k> r+j
  \end{array}
\right.
=\\
    & =
\left\{
  \begin{array}{llll}
    \mathfrak{h}\left((r+i-j,r+l-k)\right), & \hbox{if~~} j\leqslant 0 & \hbox{and}  & k\leqslant 0;\\
    \mathfrak{h}\left((r+i-j+k,r+l)\right), & \hbox{if~~} j\leqslant 0 & \hbox{and}  & k>0;\\
    \mathfrak{h}\left((r+i,r+j-k+l)\right), & \hbox{if~~} j>0          & \hbox{and}  & k\leqslant j;\\
    \mathfrak{h}\left((r+i-j+k,r+l)\right), & \hbox{if~~} j>0          & \hbox{and}  & k> j
  \end{array}
\right.
=\\
    & =
\left\{
  \begin{array}{llll}
    (r+i-j,r+l-k+p), & \hbox{if~~} j\leqslant 0 & \hbox{and}  & k\leqslant 0;\\
    (r+i-j+k,r+l+p), & \hbox{if~~} j\leqslant 0 & \hbox{and}  & k>0;\\
    (r+i,r+j-k+l+p), & \hbox{if~~} j>0          & \hbox{and}  & k\leqslant j;\\
    (r+i-j+k,r+l+p), & \hbox{if~~} j>0          & \hbox{and}  & k> j
  \end{array}
\right.
\end{split}
\end{equation*}
and
\begin{equation*}
\begin{split}
  \mathfrak{h}\left(r+i,r+j\right) &\ast_{r+p,r}\mathfrak{h}\left((r+k,r+l)\right)  
     = (r+i,r+j+p)\cdot(r+p,r)\cdot(r+k,r+l+p)=\\
    & =
\left\{
  \begin{array}{ll}
    (r+i-j,r)\cdot(r+k,r+l+p), & \hbox{if~~} r+j+p\leqslant r+p;\\
    (r+i,r+j)\cdot(r+k,r+l+p), & \hbox{if~~} r+j+p>r+p
  \end{array}
\right.
=\\
    & =
\left\{
  \begin{array}{llll}
    (r+i-j,r+l-k+p), & \hbox{if~~} r+j+p\leqslant r+p & \hbox{and}  & r+k\leqslant r;\\
    (r+i-j+k,r+l+p), & \hbox{if~~} r+j+p\leqslant r+p & \hbox{and}  & r+k>r;\\
    (r+i,r+j-k+l+p), & \hbox{if~~} r+j+p>r+p          & \hbox{and}  & r+k\leqslant r+j;\\
    (r+i-j+k,r+l+p), & \hbox{if~~} r+j+p>r+p          & \hbox{and}  & r+k> r+j
  \end{array}
\right.
=\\
    & =
\left\{
  \begin{array}{llll}
    (r+i-j,r+l-k+p), & \hbox{if~~} j\leqslant 0 & \hbox{and}  & k\leqslant 0;\\
    (r+i-j+k,r+l+p), & \hbox{if~~} j\leqslant 0 & \hbox{and}  & k>0;\\
    (r+i,r+j-k+l+p), & \hbox{if~~} j>0          & \hbox{and}  & k\leqslant j;\\
    (r+i-j+k,r+l+p), & \hbox{if~~} j>0          & \hbox{and}  & k> j,
  \end{array}
\right.
\end{split}
\end{equation*}
because $\mathfrak{h}\left((r,r)\right)=(r,r+p)$,
and hence $\mathfrak{h}\colon \mathscr{C}_{\mathbb{Z}}^{r,r}\to\mathscr{C}_{\mathbb{Z}}^{r+p,r}$ is a homomorphism. Also, the definition of the map $\mathfrak{h}$ implies that it is a bijection, and thus $\mathfrak{h}$ is an isomorphism.
\end{proof}

\begin{lemma}\label{lemma-2.7}
For an arbitrary integer $r$ and an arbitrary positive integer $p$ the variants $\mathscr{C}_{\mathbb{Z}}^{r,r}$  and $\mathscr{C}_{\mathbb{Z}}^{r,r+p}$ of the extended bicyclic semigroup $\mathscr{C}_{\mathbb{Z}}$ are isomorphic.
\end{lemma}

\begin{proof}
We define a map $\mathfrak{h}\colon \mathscr{C}_{\mathbb{Z}}^{r,r}\to\mathscr{C}_{\mathbb{Z}}^{r,r+p}$ by the formula
\begin{equation*}
\mathfrak{h}\left((r+i,r+j)\right)=(r+i,r+j+p).
\end{equation*}
The proof that so defined map $\mathfrak{h}$ is an isomorphism, is similar as in Lemma~\ref{lemma-2.6}.
\end{proof}

Lemmas~\ref{lemma-2.5}, ~\ref{lemma-2.6} and~\ref{lemma-2.7} imply the following theorem.

\begin{theorem}\label{theorem-2.8}
Any two variants  of the extended bicyclic semigroup $\mathscr{C}_{\mathbb{Z}}$ are isomorphic.
\end{theorem}

\begin{theorem}\label{theorem-2.9}
The variant $\mathscr{C}_{\mathbb{Z}}^{0,0}$  of the extended bicyclic semigroup $\mathscr{C}_{\mathbb{Z}}$ is not finitely generated.
\end{theorem}

\begin{proof} 
Formulae \eqref{eq-2.1} and \eqref{eq-2.2} imply
\begin{equation*}
  \left\{(a,b)\right\}\cup(a,b)\ast_{0,0}\mathscr{C}_{\mathbb{Z}}^{0,0}  =
   \left\{
    \begin{array}{cl}
      \left\{(x,y)\in\mathscr{C}_{\mathbb{Z}}^{0,0}\colon x\geqslant a\right\}, & \hbox{if~} b\geqslant 0;\\
      \left\{(a,b)\right\}\cup\left\{(x,y)\in\mathscr{C}_{\mathbb{Z}}^{0,0}\colon x\geqslant a-b\right\}, & \hbox{if~} b< 0
    \end{array}
    \right.
\end{equation*}
and
\begin{equation*}
\left\{(a,b)\right\}\cup\mathscr{C}_{\mathbb{Z}}^{0,0}\ast_{0,0}(a,b) =
   \left\{
    \begin{array}{cl}
      \left\{(x,y)\in\mathscr{C}_{\mathbb{Z}}^{0,0}\colon y\geqslant b\right\}, & \hbox{if~} a\geqslant 0;\\
      \left\{(a,b)\right\}\cup\left\{(x,y)\in\mathscr{C}_{\mathbb{Z}}^{0,0}\colon y\geqslant b-a\right\}, & \hbox{if~} a< 0.
    \end{array}
    \right.
\end{equation*}
Hence for every finite subset $F$ of the semigroup $\mathscr{C}_{\mathbb{Z}}^{0,0}$ we have that the set
\begin{equation*}
\big\{(x,y)\in\mathscr{C}_{\mathbb{Z}}^{0,0}\colon x,y<0\big\}\setminus \left\langle F\right\rangle
\end{equation*}
is infinite, where $\left\langle F\right\rangle$ is a subsemigroup of $\mathscr{C}_{\mathbb{Z}}^{0,0}$ generated by the set $F$, which implies the statement of the theorem.
\end{proof}

Theorems~\ref{theorem-2.8} and~\ref{theorem-2.9} imply the following corollary.

\begin{corollary}\label{corollary-2.10}
For any integers $m$ and $n$ the variant $\mathscr{C}_{\mathbb{Z}}^{m,n}$  of the extended bicyclic semigroup $\mathscr{C}_{\mathbb{Z}}$ is not finitely generated.
\end{corollary}

\section{Shift-continuous topologies on the variant $\mathscr{C}_{\mathbb{Z}}^{0,0}$}\label{section-3}

Simple calculations and formula \eqref{eq-0.1} imply the following lemma.

\begin{lemma}\label{lemma-3.1}
If $(a,b)\cdot(c,d)=(i,j)$ in $\mathscr{C}_{\mathbb{Z}}$ then $a-b+c-d=i-j$.
\end{lemma}

Lemma~\ref{lemma-3.1} implies the following proposition.

\begin{proposition}\label{proposition-3.2}
Let $m$ and $n$ be arbitrary integers. If $(a,b)\ast_{m,n}(c,d)=(i,j)$ in $\mathscr{C}_{\mathbb{Z}}^{m,n}$ then
$$
a-b+m-n+c-d=i-j.
$$
\end{proposition}

\begin{corollary}\label{corollary-3.3}
If $(a,b)\ast_{0,0}(c,d)=(i,j)$ in $\mathscr{C}_{\mathbb{Z}}^{0,0}$ then $a-b+c-d=i-j$.
\end{corollary}

Later, for every $(a,b)\in\mathscr{C}_{\mathbb{Z}}^{0,0}$ by $\lambda_{(a,b)}$ and $\rho_{(a,b)}$ we denote left and right shift (translation) on the element $(a,b)$ in the semigroup $\mathscr{C}_{\mathbb{Z}}^{0,0}$, respectively, i.e.,
$$
\lambda_{(a,b)}\colon \mathscr{C}_{\mathbb{Z}}^{0,0}\to \mathscr{C}_{\mathbb{Z}}^{0,0}, \qquad (x,y)\mapsto (a,b)\ast_{0,0}(x,y)
$$
and
$$
\rho_{(a,b)}\colon \mathscr{C}_{\mathbb{Z}}^{0,0}\to \mathscr{C}_{\mathbb{Z}}^{0,0}, \qquad (x,y)\mapsto(x,y)\ast_{0,0}(a,b).
$$

\begin{proposition}\label{proposition-3.4}
Let $\tau$ be a Hausdorff shift-continuous topology on the semigroup $\mathscr{C}_{\mathbb{Z}}^{0,0}$. Then the following assertions hold:
\begin{itemize}
  \item[$(i)$] $(a,b)$ is an isolated point in $\big(\mathscr{C}_{\mathbb{Z}}^{0,0},\tau\big)$ for any positive integers $a$ and $b$;
  \item[$(ii)$] for any integers $a$ and $b$ the set $\left\{(a-i,b-i)\colon i\in\mathbb{N}_0\right\}$ is open in $\big(\mathscr{C}_{\mathbb{Z}}^{0,0},\tau\big)$;
  \item[$(iii)$] $(a,b)$ is an isolated point in $\big(\mathscr{C}_{\mathbb{Z}}^{0,0},\tau\big)$ for any positive integer $a$ and any integer~$b$;
  \item[$(iv)$] $(a,b)$ is an isolated point in $\big(\mathscr{C}_{\mathbb{Z}}^{0,0},\tau\big)$ for any integer $a$ and any positive integer~$b$.
\end{itemize}
\end{proposition}

\begin{proof} 
$(i)$ Fix an arbitrary point $(a,b)$ in $\big(\mathscr{C}_{\mathbb{Z}}^{0,0},\tau\big)$ such that $a>0$ and $b>0$. Since by Lemma~\ref{lemma-1.1} the set $\mathscr{C}_{\mathbb{Z}}^{\geqslant 0}=\left\{(i,j)\in\mathscr{C}_{\mathbb{Z}}\colon i,j\geqslant 0\right\}$ with the induced semigroup from $\mathscr{C}_{\mathbb{Z}}$ operation is isomorphic to the bicyclic semigroup $\mathscr{C}(p,q)$ and by Proposition~1 of \cite{Bertman-West-1976} every shift-continuous Hausdorff topology on the bicyclic semigroup $\mathscr{C}(p,q)$ is discrete, there exists an open neighbourhood $U_{(a,b)}$ of the point $(a,b)$ in $\big(\mathscr{C}_{\mathbb{Z}}^{0,0},\tau\big)$ such that $U_{(a,b)}\cap \mathscr{C}_{\mathbb{Z}}^{\geqslant 0}=\left\{(a,b)\right\}$. Since
\begin{equation*}
  (i,i)\ast_{0,0}(x,y)  =(i,i)(0,0)(x,y)=   (i,i)(x,y)=
\left\{
  \begin{array}{ll}
    (i,i-x+y), & \hbox{if~} x\leqslant i;\\
    (x,y), & \hbox{if~} x>i
  \end{array}
\right.
\end{equation*}
for any non-negative integer $i$, we have that $\left\{(s,l+s-k)\colon s\leqslant k, s\in \mathbb{Z}\right\}$ is the set of solutions of the equation  $(k,l)=(k,k)\ast_{0,0}(x,y)$ for all non-negative integers $k$ and $l$. Then the separate continuity of the semigroup operation in $\big(\mathscr{C}_{\mathbb{Z}}^{0,0},\tau\big)$, Hausdorffness of $\big(\mathscr{C}_{\mathbb{Z}}^{0,0},\tau\big)$ and above arguments imply that the set
$$
\left\{(s,b+s-a)\colon s<a, s\in \mathbb{Z}\right\}=\lambda_{(a-1,a-1)}^{-1}(\{(a-1,b-1)\})
$$
is closed in $\big(\mathscr{C}_{\mathbb{Z}}^{0,0},\tau\big)$ and the set
$$
\left\{(s,b+s-a)\colon s\leqslant a, s\in \mathbb{Z}\right\}=\lambda_{(a,a)}^{-1}(U_{(a,b)})
$$
is open in $\big(\mathscr{C}_{\mathbb{Z}}^{0,0},\tau\big)$, which implies that $(a,b)$ is an isolated point in $\big(\mathscr{C}_{\mathbb{Z}}^{0,0},\tau\big)$.

$(ii)$ The proof of item $(i)$ implies that the set
$$
\left\{(s,b+s-a)\colon s\leqslant a, s\in \mathbb{Z}\right\}=\lambda_{(a,a)}^{-1}(U_{(a,b)})
$$
is open in $\big(\mathscr{C}_{\mathbb{Z}}^{0,0},\tau\big)$ for any positive integers $a$ and $b$, because there exists an open neighbourhood $U_{(a,b)}$ of the point $(a,b)$ in $\big(\mathscr{C}_{\mathbb{Z}}^{0,0},\tau\big)$ such that $U_{(a,b)}\cap \mathscr{C}_{\mathbb{Z}}^{\geqslant 0}=\left\{(a,b)\right\}$. If we put $i=a-s$ then
$$
\left\{(a-i,b-i)\colon i\in\mathbb{N}_0\right\}=\lambda_{(a,a)}^{-1}(U_{(a,b)})
$$
is an open subset of $\big(\mathscr{C}_{\mathbb{Z}}^{0,0},\tau\big)$.
It is obvious that for arbitrary integers $a$ and $b$ there exists a positive integer $k_{(a,b)}$ such that $a+k_{(a,b)}>0$ and $b+k_{(a,b)}>0$. Hausdorffness of $\big(\mathscr{C}_{\mathbb{Z}}^{0,0},\tau\big)$ implies that every point is a closed subset of $\big(\mathscr{C}_{\mathbb{Z}}^{0,0},\tau\big)$ and hence the set
$$
\left\{(a-i,b-i)\colon i\in\mathbb{N}_0\right\}=\lambda_{(a,a)}^{-1}(U_{(a,b)})\setminus\left\{(a+1,b+1),\dots,(a+k_{(a,b)},b+k_{(a,b)})\right\}
$$
is open in $\big(\mathscr{C}_{\mathbb{Z}}^{0,0},\tau\big)$, which implies the required statement.

$(iii)$ Since
\begin{equation*}
  (i,i)\ast_{0,0}(x,y) =(i,i)(0,0)(x,y)=  (i,i)(x,y)=
\left\{
  \begin{array}{ll}
    (i,i-x+y), & \hbox{if~} x\leqslant i;\\
    (x,y), & \hbox{if~} x>i
  \end{array}
\right.
\end{equation*}
for any non-negative integer $i$, we have that $\left\{(s,l+s-k)\colon s\leqslant k, s\in \mathbb{Z}\right\}$ is the set of solutions of the equation $(k,l)=(k,k)\ast_{0,0}(x,y)$ for every non-negative integer $k$ and every integer $l$. Then the separate continuity of the semigroup operation in $\big(\mathscr{C}_{\mathbb{Z}}^{0,0},\tau\big)$ and Hausdorffness of $\big(\mathscr{C}_{\mathbb{Z}}^{0,0},\tau\big)$ imply that the set
$$
\left\{(s,l+s-k)\colon s\leqslant k, s\in \mathbb{Z}\right\}=\lambda_{(k,k)}^{-1}(\{(k,l)\})
$$
is closed in $\big(\mathscr{C}_{\mathbb{Z}}^{0,0},\tau\big)$ for every non-negative integer $k$ and every integer $l$. Fix an arbitrary positive integer $a$ and an arbitrary integer $b$. Then the above arguments and assertion $(ii)$ imply that
\begin{equation*}
\{(a,b)\}=\left\{(a-i,b-i)\colon i\in\mathbb{N}_0\right\}\setminus\lambda_{(a-1,a-1)}^{-1}(\{(a-1,b-1)\})
\end{equation*}
is an open subset of $\big(\mathscr{C}_{\mathbb{Z}}^{0,0},\tau\big)$.

$(iv)$ Since
\begin{equation*}
  (x,y)\ast_{0,0}(i,i)  =(x,y)(0,0)(i,i)= (x,y)(i,i)=
\left\{
  \begin{array}{ll}
    (i-y+x,i), & \hbox{if~} y\leqslant i;\\
    (x,y), & \hbox{if~} y>i
  \end{array}
\right.
\end{equation*}
for any non-negative integer $i$, we have that $\left\{(l+s-k, s)\colon s\leqslant k, s\in \mathbb{Z}\right\}$ is the set of solutions of the equation  $(l,k)=(x,y)\ast_{0,0}(k,k)$ for every non-negative integer $k$ and every integer $l$. Then the separate continuity of the semigroup operation in $\big(\mathscr{C}_{\mathbb{Z}}^{0,0},\tau\big)$ and Hausdorffness of $\big(\mathscr{C}_{\mathbb{Z}}^{0,0},\tau\big)$ imply that the set
$$
\left\{(l+s-k, s)\colon s\leqslant k, s\in \mathbb{Z}\right\}=\rho_{(k,k)}^{-1}(\{(l,k)\})
$$
is closed in $\big(\mathscr{C}_{\mathbb{Z}}^{0,0},\tau\big)$ for every non-negative integer $k$ and every integer $l$. Fix an arbitrary integer $a$ and an arbitrary positive integer $b$. Then the above arguments and assertion $(ii)$ imply that
\begin{equation*}
\{(a,b)\}=\left\{(a-i,b-i)\colon i\in\mathbb{N}_0\right\}\setminus\rho_{(b-1,b-1)}^{-1}(\{(a-1,b-1)\})
\end{equation*}
is an open subset of $\big(\mathscr{C}_{\mathbb{Z}}^{0,0},\tau\big)$.
\end{proof}

We summarize the results of Proposition~\ref{proposition-3.4} in the following theorem.

\begin{theorem}\label{theorem-3.5}
Let $\tau$ be a Hausdorff shift-continuous topology on the semigroup $\mathscr{C}_{\mathbb{Z}}^{0,0}$. Then each of the inequalities $a>0$ or $b>0$ implies that $(a,b)$ is an isolated point of $\big(\mathscr{C}_{\mathbb{Z}}^{0,0},\tau\big)$.
\end{theorem}

The following example shows that the statement of Theorem~\ref{theorem-3.5} is complete and it cannot be extended on any point $(a,b)$ with the properties $a\leqslant 0$ and $b\leqslant 0$.

\begin{example}\label{example-3.6}
We define the topology $\tau^*$ on  $\mathscr{C}_{\mathbb{Z}}^{0,0}$ in the following way. Put
\begin{itemize}
  \item[$(i)$] $(a,b)$ is an isolated point of $\big(\mathscr{C}_{\mathbb{Z}}^{0,0},\tau^*\big)$ if and only if at least one of the following conditions holds $a>0$ or $b>0$;
  \item[$(ii)$] if $ab=0$ and $a+b\leqslant 0$ we let $A_{(a,b)}=\left\{(a-i,b-i)\colon i\in\mathbb{N}_0\right\}$ be an arbitrary Hausdorff space and $A_{(a,b)}$ be an open-and-closed subset of $\big(\mathscr{C}_{\mathbb{Z}}^{0,0},\tau^*\big)$.
\end{itemize}
It is obvious that $\big(\mathscr{C}_{\mathbb{Z}}^{0,0},\tau^*\big)$ is a Hausdorff space.
\end{example}

\begin{proposition}\label{proposition-3.7}
$\big(\mathscr{C}_{\mathbb{Z}}^{0,0},\tau^*\big)$ is a topological semigroup.
\end{proposition}

\begin{proof} 
Since $(a,b)$ is an isolated point of $\big(\mathscr{C}_{\mathbb{Z}}^{0,0},\tau^*\big)$ in the case when $a>0$ or $b>0$, it is complete to show that the semigroup operation of $\big(\mathscr{C}_{\mathbb{Z}}^{0,0},\tau^*\big)$ is continuous in the following three cases:
\begin{itemize}
  \item[(1)] $(a,b)\ast_{0,0}(c,d)$, when $a\leqslant0$, $b\leqslant 0$, $c\leqslant0$ and $d\leqslant 0$;

  \item[(2)] $(a,b)\ast_{0,0}(c,d)$, when $a\leqslant0$, $b\leqslant 0$, and $c>0$ or $d>0$;

  \item[(3)] $(a,b)\ast_{0,0}(c,d)$, when $c\leqslant0$ and $d\leqslant 0$, and $a>0$ or $b>0$.
\end{itemize}

In case (1) we have that
\begin{equation*}
  (a,b)\ast_{0,0}(c,d)=(a,b)(0,0)(c,d)=(a-b,0)(c,d)=(a-b,d-c).
\end{equation*}
Also, in this case since
\begin{equation*}
\begin{split}
  (a-i,b-i)\ast_{0,0}(c-j,d-j) & =(a-i,b-i)(0,0)(c-j,d-j)= \\
    & =(a-i-b+i,0)(c-j,d-j)=\\
    & =(a-b,0)(c-j,d-j)=\\
    & =(a-b,d-j-c+j)=\\
    & =(a-b,d-c)
\end{split}
\end{equation*}
for any $i,j\in\mathbb{N}_0$, we obtain that $A_{(a,b)}\ast_{0,0}A_{(c,d)}=\{(a-b,d-c)\}$, and hence in case (1) the semigroup operation in $\big(\mathscr{C}_{\mathbb{Z}}^{0,0},\tau^*\big)$ is continuous.

Suppose that case (2) holds. Then we have that
\begin{equation*}
\begin{split}
  (a,b)\ast_{0,0}(c,d) & =(a,b)(0,0)(c,d)= \\
    & = (a-b,0)(c,d)=\\
    & =
\left\{
  \begin{array}{ll}
    (a-b,d-c), & \hbox{if~} c\leqslant 0;\\
    (c-a+b,d), & \hbox{if~} c>0.
  \end{array}
\right.
\end{split}
\end{equation*}
In this case, since
\begin{equation*}
\begin{split}
  (a-i,b-i)\ast_{0,0}(c,d) & =(a-i,b-i)(0,0)(c,d)= \\
    & =(a-i-b+i,0)(c,d)=\\
    & =(a-b,0)(c,d)=\\
    & =
    \left\{
  \begin{array}{ll}
    (a-b,d-c), & \hbox{if~} c\leqslant 0;\\
    (c-a+b,d), & \hbox{if~} c>0,
  \end{array}
\right.
\end{split}
\end{equation*}
for every $i\in\mathbb{N}_0$ we get that
\begin{equation*}
A_{(a,b)}\ast_{0,0}\{(c,d)\}=
\left\{
  \begin{array}{ll}
    \{(a-b,d-c)\}, & \hbox{if~} c\leqslant 0;\\
    \{(c-a+b,d)\}, & \hbox{if~} c>0,
  \end{array}
\right.
\end{equation*}
which implies that the semigroup operation in $\big(\mathscr{C}_{\mathbb{Z}}^{0,0},\tau^*\big)$ is continuous  in case (2).

Suppose that case (3) holds. Then we have that
\begin{equation*}
\begin{split}
  (a,b)\ast_{0,0}(c,d) & =(a,b)(0,0)(c,d)= \\
    & =(a,b)(0,d-c)= \\
    &=
\left\{
  \begin{array}{ll}
    (a-b,d-c), & \hbox{if~} b\leqslant 0;\\
    (a,b-c+d), & \hbox{if~} b>0.
  \end{array}
\right.
\end{split}
\end{equation*}
In this case, since
\begin{equation*}
\begin{split}
  (a,b)\ast_{0,0}(c-j,d-j) & =(a,b)(0,0)(c-j,d-j)= \\
    & =(a,b)(0,d-j-c+j)=\\
    & =(a,b)(0,d-c)=\\
    & =\left\{
  \begin{array}{ll}
    (a-b,d-c), & \hbox{if~} b\leqslant 0;\\
    (a,b-c+d), & \hbox{if~} b>0,
  \end{array}
\right.
\end{split}
\end{equation*}
for every $j\in\mathbb{N}_0$ we obtain that
\begin{equation*}
\{(a,b)\}\ast_{0,0}A_{(c,d)}=
\left\{
  \begin{array}{ll}
    \{(a-b,d-c)\}, & \hbox{if~} b\leqslant 0;\\
    \{(a,b-c+d)\}, & \hbox{if~} b>0,
  \end{array}
\right.
\end{equation*}
and hence in case (3) the semigroup operation in $\big(\mathscr{C}_{\mathbb{Z}}^{0,0},\tau^*\big)$ is continuous.
\end{proof}

\begin{remark}
A topological semigroup $S$ is called
$\Gamma$-compact if for every $x\in S$ the closure of the set
$\{x,x^2,x^3,\ldots\}$ is a compactum in $S$ (see
\cite{Hildebrant-Koch-1986}). Since by Lemma~\ref{lemma-1.1} the semigroup
$\mathscr{C}_{\mathbb{Z}}^{0,0}$ contains the
bicyclic semigroup as a subsemigroup the results obtained in
\cite{Anderson-Hunter-Koch-1965}, \cite{Banakh-Dimitrova-Gutik-2009},
\cite{Banakh-Dimitrova-Gutik-2010},
\cite{Hildebrant-Koch-1986} imply that \emph{if a Hausdorff topological
semigroup $S$ satisfies one of the following conditions:
\begin{itemize}
  \item[$(i)$] $S$ is compact;
  \item[$(ii)$] $S$ is $\Gamma$-compact;
  \item[$(iii)$] the square $S\times S$ is countably compact;
  \item[$(iv)$] the square $S\times S$ is a Tychonoff pseudocompact space,
\end{itemize}
    then $S$ does not contain an algebraic copy of the semigroup $\mathscr{C}_{\mathbb{Z}}^{0,0}$.}
\end{remark}


\section*{Acknowledgements}
We acknowledge Serhii Bardyla, Alex Ravsky, and  the referee for useful important comments and suggestions.


\end{document}